\documentclass[11pt]{amsart}
\usepackage{geometry}                % See geometry.pdf to learn the layout options. There are lots.
\usepackage{graphicx}
\usepackage{amsmath,amssymb,amsthm}

\title{Frobenius numbers of Pythagorean triples}

\author[Gil]{Byung Keon Gil}
\address[Byung Keon Gil]{Naedong Middle School\\Republic of Korea}
\email{kbyungk611@naver.com}
\author[Han]{Ji-woo Han}
\address[Ji-woo Han]{Gyeonggi Science High School for the Gifted\\Republic of Korea}
\email{cindy035@naver.com}
\author[Kim]{Tae Hyun Kim}
\address[Tae Hyun Kim]{Dong Middle School\\Republic of Korea}
\email{kth508@naver.com}
\author[Koo]{Ryun Han Koo}
\address[Ryun Han Koo]{Ungnam Middle School\\Republic of Korea}
\email{gangsu999@naver.com}
\author[Lee]{Bon Woo Lee}
\address[Bon Woo Lee]{Daegu Science High School\\Republic of Korea}
\email{righthim@naver.com}
\author[Lee]{Jaehoon Lee}
\address[Jaehoon Lee]{Department of Mathematics\\University of California, Los Angeles\\United States of America}
\email{jaehoon.lee900907@gmail.com}
\author[Nam]{Kyeong Sik Nam}
\address[Kyeong Sik Nam]{Department of Mathematics\\Seoul National University\\Republic of Korea}
\email{nam-whoami@hanmail.net}
\author[Park]{Hyeon Woo Park}
\address[Hyeon Woo Park]{Daebang Middle School\\Republic of Korea}
\email{jimmy0709@naver.com}
\author[Park]{Poo-Sung Park$^\dag$}
\thanks{$^\dag$ Corresponding author}
\address[Poo-Sung Park]{Department of Mathematics Education\\Kyungnam University\\Republic of Korea}
\email[Corresponding author]{pspark@kyungnam.ac.kr}
\thanks{This research was supported by Basic Science Research Program through the National Research Foundation of Korea(NRF) funded by the Ministry of Education(NRF-2013R1A1A2010614).}
\thanks{This research was partially supported by Institute of the Gifted Education in Science of Kyungnam University.}
\newtheorem*{Lem}{Lemma}
\newtheorem*{Thm}{Theorem}

\begin{document}
\maketitle

\begin{abstract}
Given relatively prime integers $a_1, \dotsc, a_n$, the Frobenius number $g(a_1, \dotsc, a_n)$ is defined as the largest integer which cannot be expressed as $x_1 a_1 + \dotsb + x_n a_n$ with $x_i$ nonnegative integers.

In this article, we give the Frobenius number of primitive Pythagorean triples. That is,
\[
	g(m^2-n^2, 2mn, m^2+n^2) = (m-1)(m^2-n^2) + (m-1)(2mn) - (m^2 + n^2).
\] 
\end{abstract}

\section{Introduction}
The histroy of Frobenius numbers goes back to J. J. Sylvester. He proposed a problem to find the largest number for two relatively prime integers $a$ and $b$ which cannot be expressed as $ax+by$ with $x$ and $y$ nonnegative integers \cite{Sylvester}. His answer was $ab-a-b$.

F. G. Frobenius emphasized this problem and asked to study such a number for more than two integers. That is, given relatively prime integers $a_1, a_2, \dotsc, a_n$, he asked to find the largest integer $N = g(a_1, a_2, \dotsc, a_n)$ for which $a_1 x_1 + a_2 x_2 + \dotsb + a_n x_n = N$ has no nonnegative integer solutions. The number henceforth has been called \emph{Frobenius number}. 

Many mathematicians have studied Frobenius numbers and found some algorithms to compute them, nevertheless no one found a general exact formula. In actual, unlike Sylvester's elegant formula, there are no polynomial solutions for more than two integers \cite{Curtis}. In general, it is NP-hard to find the Frobenius numbers for the number of given integers \cite{Ramirez-Alfonsin}. 

So, lots of the studies have focused on specific sets of integers. For example, formulas for arithmetic progressions and geometric progressions were found \cite{Roberts}, \cite{OP}. The formula for Fibonacci numbers $F_i, F_{i+2}, F_{i+k}$ was found \cite{MRR}.

In the present article, we give a formula for primitive Pythagorean triples. This research was performed as a program of Institute of the Gifted Education in Science of Kyungnam university.

\section{Main result}

We have that $(m^2-n^2, 2mn, m^2+n^2)$ is a primitive Pythagorean triple if two integers $m > n$ are realtively prime and have opposite parity.

\begin{Thm}
	The Frobenius number for a Pythagorean triple is 
	\begin{align*}
		&g(m^2-n^2, 2mn, m^2+n^2) \\
		&= (m-1)(m^2-n^2) + (m-1)(2mn) - (m^2+n^2).
	\end{align*}
\end{Thm}

For convenience, let 
\begin{align*}
	A 
	&= (m-1)(m^2-n^2) + (m-1)(2mn) - (m^2+n^2) \\
	&= m (m^2 +2mn - n^2 - 2m - 2n).
\end{align*}

First, we prove that $g(m^2-n^2, 2mn, m^2+n^2) \ge$ A. To do this, we show that A cannot be expresssed by $m^2-n^2$, $2mn$, and $m^2+n^2$.

Suppose that there exist nonnegative integers $x$, $y$, and $z$ satisfying
\[
	x(m^2-n^2) + y(2mn) + z(m^2+n^2) = A. \tag{*}
\]

Dividing both sides by $m$, we obtain that $(-x+z)n^2 \equiv 0 \pmod{m}$. Since $m$ and $n$ are relatively prime, we can set $z = x+mk$. Plug it into the above expression and divide by $m$. Then
\begin{align*}
	km^2 + 2xm + 2yn + kn^2 &= m^2 + 2mn - n^2 - 2m - 2n \\
	k(m^2 + n^2) + 2(xm+yn) &= (m^2-n^2) + 2(mn-m-n)	
\end{align*}
and we can conclude that $k$ is odd since $m^2+n^2$ and $m^2-n^2$ are both odd.

Setting $k = 2\ell+1$, we obtain
\[
	\ell m^2 + \ell n^2 + xm + yn + n^2 = mn-m-n.	
\]

If $\ell \ge 0$, then
\[
	(\ell m + x)m + (\ell n + y + n)n = mn - m - n,
\]
but this is absurd because $g(m,n) = mn-m-n$ by Sylvester. It follows that $\ell \le -1$. Set $\ell' = -\ell-1 \ge 0$.

Changing roles of $x$ and $z$, plugging $x = z - mk$ into the above expression (*) yields
\[
	-mk(m^2-n^2) + y(2mn) + z(2m^2) = m(m^2+2mn-n^2-2m-2n).
\]
Divide both sides by $m$ and rearrange.
\[
	-k(m^2-n^2) + y(2n) +z(2m) = m^2 + 2mn - n^2 -2m - 2n
\]
Using $k = 2\ell+1 = 2(-\ell'-1)+1 = -2\ell'-1$, we can write
\begin{align*}
	\ell'(m+n)(m-n) + yn + zm &= mn - m - n \\
	\left(\ell'(m-n)+z\right)m + \left(\ell'(m-n)+y\right)n &= mn-m-n.
\end{align*}
This is also absurd. So it was proved that
\begin{align*}
	&g(m^2-n^2, 2mn, m^2+n^2) \\
	&\ge (m-1)(m^2-n^2) + (m-1)(2mn) - (m^2+n^2).
\end{align*}

Now, we prove that $g(m^2-n^2,2mn,m^2+n^2) \le A$. That is, we show that every integer greater than $A$ is expressed as $x(m^2-n^2) + y(2mn) + z(m^2+n^2)$ with $x,y,z \ge 0$. To do this, we need a lemma.

\begin{Lem}
For a fixed positive integer $b$, we define $y$ to be the smallest positive integer such that the interval $\left[\frac{b+yn}{m},\frac{ym}{n}\right]$ contains an integer, and we also define $x$ to be the smallest integer contained. Then, the inequality
\[
	(ym+xn)(m^2-n^2) \le A + (m^2-n^2) + b(2mn)
\]
holds.
\end{Lem}
\begin{proof}
The existence of $y$ is guaranteed from $\frac{n}{m} < 1 < \frac{m}{n}$.

From the condition, if $y = 0$, then $x$ and $b$ also vanish. Thus the inequlity holds for $y=0$ trivially.

Let us consider when $y \ge 1$. We have that $x-1 < \frac{b+yn}{m}$ from the minimality of $x$. Also,
\[
	\frac{b+(y-1)n}m = \frac{b+yn}m - \frac{n}{m} > \frac{b+yn}{m} - 1 > x-2.
\]
We divide two cases according to comparison of $x-1$ and $\frac{b+(y-1)n}m$.

Case 1: $x-1 < \frac{b+(y-1)n}m < x$.

Note that 
\[
	(x-1)m + 1 = xm - m + 1 \le b + (y-1)n = b + yn - n. \tag{1-1}
\]

If $x \le \frac{(y-1)m}{n}$, then $\frac{b+(y-1)n}{m} < x \le \frac{(y-1)m}{n}$ and it contradicts the minimality of $y$. Thus, $\frac{(y-1)m}{n} < x$ or 
\[
	(y-1)m+1 = ym - m + 1 \le xn. \tag{1-2}
\]

Combining the above two inequalities (1-1) and (1-2), we obtain 
\[
	m (xm-m+1+n-b) \le ymn \le n(xn+m-1).
\]
This inequality yields a new bound for $x$. That is,
\begin{equation*}
	x \le \frac{m^2+bm-m-n}{m^2-n^2}. \tag{1-3}
\end{equation*}
Now, we can conclude that
\begin{align*}
&(ym+xn)(m^2-n^2) \\
&\le (xn+m-1 + xn)(m^2-n^2) &\text{by (1-2)}\\
&\le \left( 2n \times \frac{m^2+bm-m-n}{m^2-n^2} + m - 1 \right) (m^2-n^2) &\text{by (1-3)}\\
&= b(2mn) + m(m^2-n^2) + (m-1)(2mn) - (m^2+n^2) \\
&= A + (m^2-n^2) + b(2mn).
\end{align*}

Case 2: $x-2 < \frac{b+(y-1)n}m < x-1$.

We already verified that $x-1 < \frac{b+yn}{m}$ at the start of the proof. Thus,
\[
	xm - m + 1 \le b + yn. \tag{2-1}
\]

If $x-1 \le \frac{(y-1)m}{n}$, then $\frac{b+(y-1)n}{m} \le x-1 \le \frac{(y-1)m}{n}$ and it cotradict the minimality of $y$. Thus $\frac{(y-1)m}{n} < x-1$ or 
\[
	(y-1)m + 1 = ym - m + 1 \le (x-1)n = xn - n. \tag{2-2}
\]

From the above two inequalities (2-1) and (2-2),
\[
	m (xm - m + 1 - b) \le ymn \le n(xn - n + m - 1)
\]
and thus
\[
	x \le 1+ \frac{mn+bm-m-n}{m^2-n^2}. \tag{2-3}
\]
Then,
\begin{align*}
&(ym+xn)(m^2-n^2) \\
&\le (xn+m-n-1 + xn)(m^2-n^2) &\text{by (2-2)}\\
&\le \left( 2n \times \frac{mn+bm-m-n}{m^2-n^2} + m + n -1 \right) (m^2-n^2) &\text{by (2-3)}\\
&= b(2mn) + m(m^2-n^2) + (m-1)(2mn) - (m^2+n^2) \\
&= A + (m^2-n^2) + b(2mn).
\end{align*}

Hence, the proof of Lemma was completed.
\end{proof}

Let us resume the proof of Theorem. Let $k$ be an arbitrary positive integer. We can choose two positive integers $a$ and $b$ such that
\[
	a(m^2-n^2) - b(2mn) = A+k.
\]
If $a$ is chosen to be negative, consider
\[
	(a + 2mnt)(m^2-n^2) - \left(b + (m^2-n^2)t \right)(2mn) = A+k
\]
for sufficiently large $t$.

Recall $x$ and $y$ in Lemma. From $\frac{b+yn}{m} \le x \le \frac{ym}{n}$,
\[
	-b + xm - yn \ge 0
	\qquad\text{ and }\qquad
	ym - xn \ge 0.
\]

If $a+1 \le ym + xn$, then
\begin{align*}
1 
&\le k = a(m^2-n^2) - b(2mn) - A \\
&\le (ym+xn-1)(m^2-n^2) - b(2mn) - A \\
&= (ym+xn)(m^2-n^2)-(m^2-n^2) - b(2mn) - A \\
&\le A + (m^2-n^2) + b(2mn) - (m^2-n^2) - b(2mn) - A &\text{by Lemma}\\
&= 0.
\end{align*}
This contradiction yields $ym + xn \le a$.

Theferefore,
\begin{align*}
A+k 
&= a(m^2-n^2)+(-b)(2mn) \\
&= (a-ym-xn)(m^2-n^2) \\
&+ (-b+xm-yn)(2mn) \\
&+ (ym-xn)(m^2+n^2)
\end{align*}
and $a-ym-xn \ge 0$, $-b+xm-yn \ge 0$, $ym-xn \ge 0$.

We conclude that $g(m^2-n^2,2mn,m^2+n^2) = A$.

\section*{Acknowledgement}

The authors would like to thank professor Byungchan Kim for reading this article and for his helpful comments.

\end{document}